\documentclass[12pt]{article}

\usepackage[authoryear]{natbib}
\usepackage{amssymb, bm, verbatim}
\usepackage{amssymb,amsthm,epsfig}
\usepackage[centertags]{amsmath}
\usepackage{times}
\usepackage{color}

\newtheorem{theorem}{{\sc Theorem}}[section]
\newtheorem{definition}{{\sc Definition}}[section]
\newtheorem{proposition}{{\sc Proposition}}[section]
\newtheorem{example}{{\sc Example}}[section]

\title{A note on Carath\'eodory's Extension Theorem}
\author{Alexandre G.~Patriota\\
{\small {\em Departamento de Estat\'istica, Universidade de S\~ao Paulo,
S\~ao Paulo/SP, 05508-090, Brazil}}}
\date{}

\begin{document}
\maketitle

\begin{abstract}

In this note, we show that the Carath\'eodory's extension theorem is still valid for  a class of subsets of $\Omega$ less restricted than a semi-ring, which we call quasi-semi-ring.
\\


\end{abstract}

\section{Introduction}

Due to paradoxes such as the Banach-Tarski paradox \citep[see][]{Banach}, it is not always possible to define a measure (e.g., Lebesgue measure) in the power set of the main set $\Omega$. Instead, we must restrict our attention to certain measurable subsets of $\Omega$.
The Carath\'eodory's extension theorem basically extends a countably additive premeasure defined in a small class, usually a 
semi-ring, to a large class of measurable sets that contains the smaller one. 
 The real line is the main motivation for using a semi-ring as the starting class of subsets, because the Borel sigma-algebra can be 
 generated by a class of semi-open intervals, which is a semi-ring. Therefore, by defining a premeasure on this class of semi-open 
 intervals (which is an easy task), an extension to the Borel sigma-algebra (which contains ``non-pathological'' subsets of $\mathbb{R}$) is readily available through the extension theorem. However, as a semi-ring requires closure by intersections, it may be more difficult to define a semi-ring of subsets of some non-flat surfaces such as cylinders and closed surfaces (sphere, torus, double torus, triple torus, Klein bottle and so on).
 
In this note, we show that it is possible to weaken the assumptions regarding the initial class of subsets in the Carath\'eodory's extension theorem. We define a new class of subsets that does not require closure by intersections and prove that: (1) all elements in this collection are measurable (in the sense of Carath\'eodory's ``splitting principle''), (2) the extension (the outer measure) agrees with the premeasure on the starting collection and (3) it is unique on the smallest ring generated by this collection. Some of the proofs given in this note are similar to those in \cite{Athreya}.

Below we define a quasi-semi-ring of subsets which plays an important role in the construction of our theory.

\begin{definition} Let $\Omega$ be nonempty. A class $\mathcal{A}$ of subsets of $\Omega$ is a quasi-semi-ring if the following conditions hold;
\begin{enumerate}
\item $\varnothing \in \mathcal{A}$,
\item If $A, B \in \mathcal{A}$, then there exist disjoint subsets $B_1, \ldots B_n, C_1, \ldots, C_k \in \mathcal{A}$ such that $A \cap B = \bigcup_{i=1}^n B_i$ and $A \cap B^c = \bigcup_{i=1}^k C_i$, where $n,k < \infty$,
\end{enumerate}
\end{definition}

The main difference between a quasi-semi-ring and a semi-ring is that the former may  not  be closed by finite intersections but the latter must be. It is not hard to see, by the above definition, that a semi-ring is always a quasi-semi-ring, but the converse is not always true. Below, we show some classes of subsets of $\Omega$ that are quasi-semi-rings but are not semi-rings. The first two examples are artificial ones, but the last one is more natural. The readers are invited to find other examples.

\begin{example}
Consider that $A, B, C \subset \Omega$ and $\mathcal{A} = \{\varnothing, A, B, A \cap B \cap C, A \cap B \cap C^c, A \cap B^c \cap C, A \cap B^c \cap C^c, A^c \cap B \cap C, A^c \cap B \cap C^c \}$. Then, $\mathcal{A}$ is a quasi-semi-ring but it is not necessarily a semi-ring (it is not closed under finite intersections). 
\end{example}

In order to better understand the above example, the reader should draw a Venn diagram with the sets $A, B$ and $C$.

\begin{example} Suppose that $\Omega = \mathbb{R}^2$ and $\mathcal{A} = \{\mbox{all semi-closed rectangles where base $\neq$ height}\} \cup{\varnothing}$. It is not a semi-ring, because some intersections of rectangles do produce squares. On the other hand, every square can be represented by finite union of disjoint rectangles with different base and height.
 Note also that if $A, B \in \mathcal{A}$, then $A \cap B$ and $A \cap B^c$ may be $\varnothing$, rectangles with different base and height or finite unions of disjoint rectangles with different base and heigh. Therefore, $\mathcal{A}$ is a quasi-semi-ring.
\end{example}

\begin{example} Let $\Omega$ be a circle in the plane and assume that  $ \mathcal{A}$ is a class containing all the semi-closed  arcs of $\Omega$, assume also that $\varnothing \in \mathcal{A}$ . It is easy to see that $\mathcal{A}$ is not a semi-ring, since it is not closed under intersections. Take the parametrized arcs $A = (0, \frac{3\pi}{2}]\in \mathcal{A}$ and $B=(\pi, \frac{5\pi}{2}]\in \mathcal{A}$, then $A \cap B = (\pi , \frac{3\pi}{2}] \cup (0, \frac{\pi}{2}] \notin \mathcal{A}$. On the other hand, it is a quasi-semi-ring, because if $A,B \in \mathcal{A}$, then
$A \cap B$ and $A\cap B^c$ are unions of semi-closed  disjoint arcs or $\varnothing$ or they are in $\mathcal{A}$. Notice that, $\mathcal{A}$ would be a semi-ring if it were defined as the  class containing all the semi-closed parametrized arcs of $\Omega$ restricted to the interval $(0, 2\pi]$. The quasi-semi-ring does not require such a restriction.
\end{example}

Apparently, a collection of subsets formed by all semi-closed ``pieces'' of any smooth close surface is not a semi-ring (since some intersections are not semi-closed ``pieces''), but, on the other hand, it is a quasi-semi-ring (since these intersections are formed by union of semi-closed ``pieces'').

With the purpose of proving our results, we use the usual tools firstly introduced in \cite{Cara}, namely: the outer measure and the Carath\'eodory's ``splitting principle'' (the criteria for measurability of sets).
Let $\Omega$ be nonempty. Given a premeasure $\mu$ well-defined in a quasi-semi-ring $\mathcal{A}$ (i.e., $\mu(\varnothing) = 0$ and it is countably additive),  the outer measure induced by $\mu$ as a function of sets from the power set  of \ $\Omega$ to $[0,\infty]$ is usually defined as
$\mu^*(A) = \inf \bigg\{ \sum_{j\geq 1} \mu(A_j) : \  A \subset \bigcup_{j\geq 1} A_j, \ \{A_j\}_{j\geq 1} \subset \mathcal{A}\bigg\}
$ for all $A \subset \Omega$.
 In this definition, it should be clear that the covers of $A$ have to be formed by countable many sets. The well-known properties of an outer measure are:  (i) $\mu^*(\varnothing) = 0$, (ii) $\mu^*$ is monotone and (iii) $\mu^*$ is countably subadditive. 

In this note, we use another equivalent definition of outer measure where the covers are formed by  disjoint sets. This will help us to prove that the outer measure equals the premeasure on the quasi-semi-ring.

\begin{proposition}\label{Alter}
The outer measure induced by $\mu$ can alternatively be defined as
\[
\bar{\mu}(A) = \inf \bigg\{ \sum_{j\geq 1} \mu(A_j) : \  A \subset \bigcup_{j\geq 1} A_j, \ \{A_j\}_{j\geq 1} \subset \mathcal{A} \ \mbox{disjoint}\bigg\}
\] for all $A \subset \Omega$.
\end{proposition}
\begin{proof}
Notice that $\mu^*(A)\leq \bar{\mu}(A)$ for all $A \subset \Omega$, since
\[
\bigg\{\{A_j\}_{j\geq 1} \subset \mathcal{A} \ \mbox{disjoint}: \  A \subset \bigcup_{j\geq 1} A_j \bigg\} \subset \bigg\{\{A_j\}_{j\geq 1} \subset \mathcal{A}: \  A \subset \bigcup_{j\geq 1} A_j \bigg\}. 
\]
Define $B_1 = A_1$, $B_2 = A_2 \cap A_1^c$, $B_i = A_i \cap A_{i-1}^c \cap \ldots \cap A_1^c$, for $i \geq 1$. By definition of quasi-semi-rings, there exist disjoint sets $C_1^n, \ldots, C_{k_n}^n \in \mathcal{A}$ such that $A_n \cap A_{n-1}^c = \bigcup_{i=1}^{k_n} C_i^{n}$. Note that $A_{n} \cap A_{n-1}^c \cap A_{n-2}^c = \bigcup_{i=1}^{k_n} (C_i^n \cap A_{n-2}^c)$, therefore exist another disjoint sequence  $D_{1,i}^n, \ldots, D_{l_{n,i}, i}^n \in \mathcal{A}$ such that $(C_i^n \cap A_{n-2}^c)= \bigcup_{j=1}^{l_{i,n}} D_{j, i}^n$, then $A_{n} \cap A_{n-1}^c \cap A_{n-2}^c = \bigcup_{i=1}^{k_n} \bigcup_{j=1}^{l_{i,n}} D_{j, i}^n$. Thus, by repetitively applying this argument, we achieve at $B_n = \bigcup_{i = 1}^{m_n} H_i^n$ such that $H_1^n, \ldots, H_{m_n}^n \in \mathcal{A}$ are disjoint sets. As $\bigcup_{n \geq 1} A_n = \bigcup_{n \geq 1} B_n = \bigcup_{n \geq 1} \bigcup_{i = 1}^{m_n} H_i^n$, we have that for each cover of $A$, $\{A_n\}_{n \geq 1}$, there exist another cover of $A$ formed by disjoint sets $\{\{H_i^n\}_{i = 1}^{m_n}\}_{n \geq 1} \in \mathcal{A}$. Also, observe that  \[A_n = \bigg(A_n \cap \bigg(\bigcup_{i=1}^{n-1}A_i\bigg)^c\bigg) \cup \bigg(A_n \cap \bigg(\bigcup_{i=1}^{n-1}A_i\bigg)\bigg) = \bigg(\bigcup_{i=1}^{m_n} H_i^n\bigg) \cup \bigg( \bigcup_{i=1}^{n-1}( A_n \cap A_i)\bigg),\]
there exist  disjoint sets  $N_1^{ni}, \ldots, N_{k_{ni}}^{ni} \in \mathcal{A}$ such that $A_n\cap A_i = \bigcup_{j=1}^{k_{ni}} N_j^{ni}$, then by finite additivity $\mu(A_n) = \sum_{i=1}^{m_n}\mu(H_i^n) + \sum_{i=1}^{n-1}\sum_{j=1}^{k_{ni}}\mu(N_j^{ni})$, thus \[\mu(A_n) \geq \sum_{i=1}^{m_n}\mu(H_i^n) \quad \mbox{and} \quad \sum_{n\geq1}\mu(A_n) \geq \sum_{n\geq1}\sum_{i=1}^{m_n}\mu(H_i^n).\] Therefore, if $A \subset \Omega$, then $\bar{\mu}(A) \leq \mu^*(A)$ and we conclude that  $\bar{\mu}(A)= \mu^*(A)$ for all $A \subset \Omega$.
\end{proof}

In what follows, we present the Carath\'eodory's ``splitting principle'' which defines measurable sets.
A set $A \subset \Omega$ is said to be $\mu^*$-mensurable if for all $E \subset \Omega$,
$ \mu^*(E) = \mu^*(E \cap A) + \mu^*(E \cap A^c).$ The class of all measurable subsets $\mathcal{M} = \{A; \ A \mbox{ is  $\mu^*$-measurable}\}$ is indeed a sigma-algebra and the triplet $(\Omega, \mathcal{M}, \mu^*)$ is a measure space independently of the starting class of subsets $\mathcal{A}$ (the premeasure $\mu$ must be countably additive). 

\subsection{Extension Theorem}
This section establishes  that all elements listed in a quasi-semi-ring are measurable and also that the outer measure is equivalent to the premeasure on the quasi-semi-ring.

\begin{theorem}\label{Main}(Extension theorem) Let $\Omega$ be nonempty, $\mathcal{A}$ a quasi-semi-ring of $\Omega$ and $\mu$ a countably additive premeasure on $\mathcal{A}$. Then,
\begin{enumerate}
\item $\mathcal{A} \subset \mathcal{M}$,
\item $\mu^*(A) = \mu(A)$ for all  $A \in \mathcal{A}$.
\end{enumerate}
\end{theorem}

\begin{proof} Let $A \in \mathcal{A}$, $E \subset \Omega$  and $\{A_j\}_{j\geq 1} \subset \mathcal{A}$ such that $E \subset \bigcup_{j\geq 1} A_j$. Notice that 
\[A_j  = (A_j \cap A ) \cup (A_j \cap A^c).\]
By definition of quasi-semi-ring, for each $j\geq 1$ there exist disjoint sets $B_{1}^j, \ldots, B_{k_j}^j  , C_{1}^j, \ldots C_{n_j}^j \in \mathcal{A}$ such that $A_j \cap A = \bigcup_{i=1}^{k_j} B_{i}^j$ and  $A_j \cap A^c = \bigcup_{i=1}^{n_j} C_{i}^j$. Therefore,
\[
A_j = \bigg(\bigcup_{i=1}^{k_j} B_{i}^j \bigg) \cup \bigg( \bigcup_{i=1}^{n_j} C_{i}^j\bigg).
\]
By the finite additivity of the premeasure $\mu$ we have
\[
\mu(A_j) = \sum_{i=1}^{k_j} \mu(B_{i}^j) + \sum_{i=1}^{n_j} \mu(C_{i}^j)
\] and
\[
\sum_{j\geq 1}\mu(A_j) = \sum_{j\geq 1} \sum_{i=1}^{k_j} \mu(B_{i}^j) + \sum_{j\geq 1}\sum_{i=1}^{n_j} \mu(C_{i}^j).
\]
Notice that $E\cap A \subset \bigcup_{j\geq 1} \bigcup_{i=1}^{k_j} B_{i}^j$  and $E \cap A^c \subset \bigcup_{j\geq 1} \bigcup_{i=1}^{n_j} C_{i}^j$, hence, by definition,
\[
\sum_{j\geq 1}\mu(A_j) \geq \mu^*(E \cap A) + \mu^*(E\cap A^c),
\] for all covers, $\{A_j\}_{j\geq1}\subset \mathcal{A}$, of $E$. Then,
\[
\mu^*(E)\geq \mu^*(E \cap A) + \mu^*(E\cap A^c).
\]
By subadditivity we conclude that $\mu^*(E) = \mu^*(E \cap A) + \mu^*(E\cap A^c)$. That is, $\forall \ A \in \mathcal{A} \Rightarrow A \in \mathcal{M}$, thus $\mathcal{A} \subset \mathcal{M}$, which proves item 1.

Now, let $A \in \mathcal{A}$ and suppose that $\mu^*(A) < \infty$ (if it is infinity the equality is obvious). 
 For each $\epsilon >0$, there exist a cover of $A$, $\{A_n\}_{n\geq 1} \subset \mathcal{A}$ such that
\begin{equation}\label{Def-inf}
\mu^*(A) \leq \sum_{n\geq 1} \mu(A_n) \leq \mu^*(A)+ \epsilon.
\end{equation}

Without lost of generality, consider that the cover of $A$ is formed by disjoint sets (see Proposition \ref{Alter}). Note that $A = \bigcup_{n\geq 1}(A \cap A_n)$, then there exist disjoint sets $M_{1}^n, \ldots, M_{r_n}^n \in \mathcal{A}$ such that $A \cap A_n = \bigcup_{i=1}^{r_n}M_{i}^n$. Therefore, by the countably additive property, we have that:
\[
\mu(A) = \mu\bigg( \bigcup_{n\geq 1}\bigcup_{i=1}^{r_n} M_{i}^n \bigg) =  \sum_{n\geq 1}\sum_{i=1}^{r_n} \mu(M_{i}^n).
\] On the other hand, there exist also disjoint sets $N_{1}^n, \ldots, N_{w_n}^n\in \mathcal{A}$ such that $A_n \cap A^c = \bigcup_{i=1}^{w_n}N_{i}^n$, thus
\[
A_n = (A_n \cap A) \cup (A_n \cap A^c) = \bigg(\bigcup_{i=1}^{r_n}M_{i}^n\bigg) \cup \bigg(\bigcup_{i=1}^{w_n}N_{i}^n\bigg).
\] By finite additivity of the measure,
\[
\sum_{n\geq 1 } \mu(A_n) = \sum_{n\geq 1 } \sum_{i=1}^{r_n}\mu(M_{i}^n) + \sum_{n\geq 1 } \sum_{i=1}^{w_n}\mu(N_{i}^n) \geq \sum_{n\geq 1 } \sum_{i=1}^{r_n}\mu(M_{i}^n)
\] and
\[
\mu(A) \leq \sum_{n\geq 1 } \mu(A_n).
\]
Then, for each $\epsilon > 0$, 
\[\mu(A) \leq \sum_{n\geq 1} \mu(A_n) \leq \mu^*(A)+ \epsilon\] implying that $\mu(A) \leq \mu^*(A)$. We conclude that $\mu(A) = \mu^*(A)$ for all  $A \in \mathcal{A}$.
\end{proof}

\subsection{Uniqueness of the extension}

As a quasi-semi-ring is not a $\pi$-system, the uniqueness of the extension is not guaranteed. In this section, we prove that the extension is unique when restricted to the smallest ring generated by the quasi-semi-ring.

\begin{proposition} If $\mathcal{A}$ is a quasi-semi-ring, then $r(\mathcal{A}) = \{A: \ A = \cup_{i=1}^k B_i, \ \{B_i\}_{i=1}^k \in \mathcal{A} \ \mbox{disjoint}\}$ is the smallest ring generated by $\mathcal{A}$.
\end{proposition}
\begin{proof} By definition of quasi-semi-ring, if $A,B \in \mathcal{A}$, then there exist disjoint sequences $A_1, \ldots, A_k,$ $ B_1, \ldots, B_n \in \mathcal{A}$ such that $A \cap B = \bigcup_{i=1}^k A_i$ and $A \cap B^c = \bigcup_{i=1}^n B_i$. By construction, $\mathcal{A} \subset r(\mathcal{A})$, thus $A, B, A \cap B, A \cap B^c \in r(\mathcal{A}) \Rightarrow  A \cup B \in r(\mathcal{A})$ (notice that $A$ and $B$ need not be disjoint sets). 

Now, let $A, B \in r(\mathcal{A})$, then there exist disjoint sets $C_1, \ldots, C_k\in \mathcal{A}$ and $D_1, \ldots, D_n\in \mathcal{A}$ such that $A = \bigcup_{i=1}^k C_i$  and $B = \bigcup_{i=1}^n D_i$  (where $C_i$ and $D_j$ need not be disjoint).
 Notice that $A \cup B =  \bigcup_{j=1}^n \bigcup_{i=1}^k (C_i \cup D_j) \in r(\mathcal{A})$, i.e., $r(\mathcal{A})$ is closed under finite unions. Note also that $A \cap B= \bigcup_{j=1}^n \bigcup_{i=1}^k (C_i \cap D_j)$ with $C_i \cap D_j \in r(\mathcal{A})$ for all $i=1, \ldots, k$ and $j=1, \ldots,n$, then $A \cap B \in r(\mathcal{A})$, since $r(\mathcal{A})$ is closed under finite unions. Finally, as $A \cap B^c = \bigcup_{i=1}^k (C_i \cap D_n^c \cap D_{n-1}^c \cap \ldots \cap D_1^c)$ and  $C_i \cap D_n^c \cap D_{n-1}^c \cap \ldots \cap D_1^c \in r(\mathcal{A})$ for all $i=1, \ldots, k$ we have that $A \cap B^c \in r(\mathcal{A})$. We conclude that $r(\mathcal{A})$ is a ring generated by $\mathcal{A}$. In fact, $r(\mathcal{A})$ is the smallest ring generated by $\mathcal{A}$, since all other rings generated by $\mathcal{A}$ must be closed under finite unions of sets from $\mathcal{A}$.
\end{proof}


\begin{theorem}\label{Uniqueness-ring} Let $\mathcal{A}$ be a quasi-semi-ring of $\Omega$. Let $\mu_1$ and $\mu_2$ be two countably additive and finite measures defined on $\mathcal{M}$ such that $\mu_1(A) = \mu_2(A)$ for all $A \in \mathcal{A}$. Then, $\mu_1(A) = \mu_2(A)$ for all $A \in r(\mathcal{A})$.
\end{theorem}
\begin{proof} Define $\mathcal{C} = \{A \in r(\mathcal{A}), \ \mu_1(A) = \mu_2(A)\}$, then $\mathcal{A} \subset \mathcal{C} \subset r(\mathcal{A})$.
 Let $A \in r(\mathcal{A})$, then $A = \bigcup_{i=1}^k D_i$, with $D_1, \ldots, D_k\in \mathcal{A}$ being disjoint sets. Notice that $D_1, \ldots, D_k \in \mathcal{C}$, then
$\mu_1(D_i) = \mu_2(D_i)$ for all $i = 1,\ldots, k$ and (by  additive property of the involved measures)
\[\mu_1(A) =  \mu_1\bigg(\bigcup_{i=1}^k D_i\bigg) = \sum_{i=1}^k\mu_1(D_i) = \sum_{i=1}^k\mu_2(D_i)=\mu_2\bigg(\bigcup_{i=1}^k D_i\bigg) = \mu_2(A)\]~
 implying that $A=\bigcup_{i=1}^k D_i \in \mathcal{C}$, therefore, $ r(\mathcal{A}) \subset \mathcal{C}$ and $r(\mathcal{A}) = \mathcal{C}$.
\end{proof}

Next, we establish that if $\Omega$ is covered by elementary sets in $\mathcal{A}$ of finite premeasure, then every set in $\mathcal{A}$ can be represented by a union of disjoint sets in $\mathcal{A}$ with also finite premeasure. 

\begin{proposition}\label{Ex} Let $\mathcal{A}$ be a quasi-semi-ring and $\mu$ a premeasure defined in $\mathcal{A}$. Assume that there exist a cover $\{A_i\}_{i \geq 1}\subset \mathcal{A}$ for $\Omega$ such that $\mu(A_i)< \infty$ for all $i \geq 1$. Then, for all $A \in \mathcal{A}$, there exist disjoint sets $\{C_i\}_{i \geq 1}\subset \mathcal{A}$ such that $A = \bigcup_{i \geq 1}C_i$ with $\mu(C_i) < \infty$ for $i \geq 1$.
\end{proposition}

\begin{proof}
Let  $\{A_i\}_{i \geq 1}\subset \mathcal{A}$ such that $\Omega = \bigcup_{i \geq 1} A_i$ with $\mu(A_i)<\infty$ for $i \geq 1$. By arguments given in Proposition \ref{Alter}, we can consider $\{A_i\}_{i \geq 1}$ disjoint sets. 

If $A \in \mathcal{A}$, then $A =  \bigcup_{i \geq 1}( A_i \cap A)$, since $A \subset \bigcup_{i \geq 1}A_i$. There exist a disjoint sequence $B_{i,1}, \ldots, B_{i,k_i} \in \mathcal{A}$ such that $A_i \cap A = \bigcup_{j = 1}^{k_i} B_{i,j}$. By monotonicity, we have that $\mu(B_{i,j})<\infty$ for all $i,j$ (since $B_{i,j} \subset A_i$ for all $i,j$). Therefore, exist a disjoint sequence of sets $\{\{B_{i,j}\}_{j=1}^{k_i}\}_{i \geq 1} \in \mathcal{A}$ such that
\[
A =  \bigcup_{i \geq 1} \bigcup_{j =1}^{k_i} B_{i,j}\quad \mbox{and}\quad \mu(B_{i,j})< \infty
\] for all $j=1, \ldots,k_i$ and  $i\geq 1$.
\end{proof}
Now, we can extend Theorem \ref{Uniqueness-ring} to the case of sigma-finite measures.

\begin{theorem}(Uniqueness theorem)  Let $\mathcal{A}$ be a quasi-semi-ring of $\Omega$ and $\mu$ a (countably additive) sigma-finite measure (i.e., there exist $\{A_i\}_{i \geq 1}\subset \mathcal{A}$ such that $\Omega = \bigcup_{i \geq 1} A_i$ with $\mu(A_i)<\infty$ for $i \geq 1$). Then, $\mu^*$ is the unique extension on $r(\mathcal{A})$ that agrees with $\mu$ on $\mathcal{A}$.
\end{theorem}
\begin{proof}
Suppose that there exist another measure $\nu$ on $r(\mathcal{A})$ that agrees with $\mu$ on $\mathcal{A}$. By Proposition \ref{Ex}, every set in $\mathcal{A}$ can be expressed as union of disjoint sets in $\mathcal{A}$ of finite premeasure. By assumption, $\nu$ and $\mu^*$ must agree for each one of these elementary sets in $\mathcal{A}$ with finite premeasure, by countably additive, we conclude that $\nu$ agree with $\mu^*$ for every set in the ring (since the sets in $r(\mathcal{A})$ can be represented by finite unions of disjoint elementary sets in $\mathcal{A}$ of finite measures).
\end{proof}

The smallest sigma-algebra generated by $\mathcal{A}$ is the smallest sigma-algebra generated by $r(\mathcal{A})$. It is known that if two sigma-finite measures agree on the ring $r(\mathcal{A})$ they must agree on the smallest sigma algebra generated by $r(\mathcal{A})$. Therefore, two sigma-finite premeasures defined on $\mathcal{A}$ must agree in the smallest sigma-algebra generated by $\mathcal{A}$.


\section*{Acknowledgments}

I gratefully acknowledge grants from FAPESP (Brazil). I wish to thank Professor Nelson Ithiro Tanaka for valuable comments and discussions on this manuscript.

\end{document}